\documentclass[12pt]{amsart}

\usepackage{amsfonts,amsmath,amssymb}
\usepackage[latin1]{inputenc}
\usepackage{graphicx}
\usepackage{xcolor}
\usepackage{colortbl}
\usepackage[curve,knot,graph,dvips]{xy}

\title{Squarefree monomial modules and \\ extremal Betti numbers}

\newtheorem{Lem}{Lemma}[section]
\newtheorem{Prop}[Lem]{Proposition}
\newtheorem{Cor}[Lem]{Corollary}
\newtheorem{Thm}[Lem]{Theorem}

\theoremstyle{definition}
\newtheorem{Def}[Lem]{Definition}
\newtheorem{Rem}[Lem]{Remark}

\newtheorem{Expl}[Lem]{Example}
\newtheorem{Expls}[Lem]{Examples}

\newtheorem{Constr}[Lem]{Construction}
\newcommand\pf{\begin{proof}}
\newcommand\epf{\end{proof}}

\def\NZQ{\Bbb}               
\def\NN{{\NZQ N}}

\def\ZZ{{\NZQ Z}}

\newcommand\indeg{{\operatorname{indeg}}}

\newcommand\supp{{\operatorname{supp}}}
\newcommand\m{{\operatorname{m}}}

\newcommand\bb{\bold{b}}
\newcommand\bd{\bold{ds}}
\newcommand\bl{\bold{dl}}

\numberwithin{equation}{section}

\begin{document}
\begin{flushright}
\emph{To appear in Algebra Colloquium}
\end{flushright}
\bigskip

\maketitle
\centerline{\scshape Marilena Crupi \and Carmela Ferr\`{o} }
\medskip
{\footnotesize
\centerline{University of Messina, Department of Mathematics and Computer Science}
\centerline{Viale Ferdinando Stagno d'Alcontres, 31}
\centerline{98166 Messina, Italy}
\centerline{E-mail: \email{mcrupi@unime.it; cferro@unime.it}}

}

\maketitle
\begin{abstract}
Let $K$ be a field and let $S = K[x_1,\dots,x_n]$ be a
polynomial ring over $K$. Let $F = \oplus_{i=1}^r Se_i$ be a
finitely generated graded free $S$-module with basis $e_1, \dots,
e_r$ in degrees $f_1, \dots, f_r$ such that
$f_1 \leq f_2 \leq \dots \leq f_r$. We examine some classes of squarefree monomial submodules of $F$. Hence, we focalize our attention on the Betti table of such classes in order to analyze the behavior of their
extremal Betti numbers.\\
{\bf 2010 Mathematics Subject Classification:} 05E40, 13B25, 13D02 16W50.\\
{\bf Keywords:} graded modules, squarefree monomial modules, minimal graded resolutions.

\end{abstract}

\section*{Introduction}
Let $K$ be a field and let $S = K[x_1,\dots,x_n]$ be the
polynomial ring in $n$ variables with coefficients in $K$. Let $F = \oplus_{i=1}^r Se_i$ be a finitely generated graded free
$S$-module with basis $e_1, \dots, e_r$ in degrees $f_1, \dots,
f_r$, renumbered as necessary so that $f_1 \leq f_2 \leq \dots
\leq f_r$.

A \textit{monomial submodule} $M$ of $F$ is a submodule generated by
monomials, \textit{i.e.}, $M = \oplus_{i=1}^r I_i e_i$, where $I_i$
are the monomial ideals of $S$ generated by those monomials $m$
of $S$ such that $me_i \in M$. A monomial submodule $M = \oplus_{i=1}^r I_i e_i \subsetneq F$ is a \textit{squarefree monomial submodule} if every $I_i$ is a squarefree monomial ideal of $S$.

If $F=S$, then a squarefree monomial submodule of $F$ is a squarefree monomial ideal of $S$. The class of squarefree monomial ideals has been studied from viewpoint of commutative algebra and combinatorics (see, for example \cite{AHH2, AHH3}).

In this paper, we are interested in the study of some classes of squarefree monomial submodules: \textit{squarefree stable submodules}, \textit{squarefree strongly stable submodules} and \textit{squarefree lexicographic submodules}.

As in the theory of squarefree monomial ideals, we have the following hierarchy of squarefree monomial submodules: squarefree lexicographic submodule $\Rightarrow$ squarefree strongly stable submodule $\Rightarrow$ squarefree stable submodule.

In \cite{AHH2}, Aramova, Herzog and Hibi constructed the explicit minimal graded free resolution of a squarefree stable ideal, similar to the Eliahou-Kervaire resolution of a stable ideal \cite{EK}, and stated a formula for computing the
Betti numbers of such class of ideals. Such formula is a fundamental tool if one wants to study the Betti table of a squarefree stable submodule.

If $M$ is a finitely generated graded $S$-module, a Betti number $\beta_{k,k+\ell}(M) \neq 0$ is called
{\it extremal} if $\beta_{i, i+j}(M) = 0$ for all $i \geq k$, $j
\geq \ell$, $(i, j) \neq (k, \ell)$.

The extremal Betti numbers were
introduced by Bayer, Charalambous and Popescu in \cite{BCP} as a refinement of two invariants of the module $M$: the projective
dimension and the Castelnuovo-Mumford regularity.

Indeed, if $\beta_{k_1,k_1+\ell_1}(M), \ldots,
\beta_{k_t,k_t+\ell_t}(M),\,\,\,k_1 > \cdots > k_t, \ell_1 <\cdots <
\ell_t$, are all extremal Betti numbers of $M$, then $k_1 =
\textrm{projdim}_S(M)$ and $\ell_t = \textrm{reg}_S(M)$.

The behavior of the extremal Betti numbers for some
classes of monomial ideals in polynomial rings in a finite
numbers of variables over a field was studied by Crupi and Utano in \cite{CU1,CU2}.
Subsequently, the same authors \cite{CU3} examined such special
graded Betti numbers for a lexicographic submodule $M$ of a
finitely generated graded free $S$-module.

In this paper we devote our attention on the extremal Betti numbers of a squarefree stable submodule. We characterize the extremal Betti numbers of such class of monomial submodules and give a criterion for determining their positions and their number.

The plan of the paper is the following. In Section \ref{pre}, some notions that  will be used throughout the
paper are recalled. In Section \ref{classes}, the classes of squarefree stable submodules, squarefree strongly stable submodules and
squarefree lexicographicsubmodules are examined and some relevant properties discussed.
In Section \ref{extr}, the behavior of the extremal Betti numbers for the squarefree stable submodules is studied; the characterization of these invariants is the main result. In Section \ref{crit}, given a squarefree stable submodule $M$ of the free $S$-module $S^r$, $r\geq 1$, a criterion to recognize among all the graded Betti numbers of $M$ those extremal
ones is given.
Section \ref{appl} contains an application on the so called \textit{super extremal Betti numbers} of a squarefree lexicographic submodule.

\section{Preliminary and notation} \label{pre}
Throughout this paper, let $S = K[x_1,\dots,x_n]$ be the polynomial ring in $n$ variables over a field
$K$ and $F = \oplus_{i=1}^r Se_i$ a finitely generated graded free
$S$-module with basis $e_1, \dots, e_r$ in degrees $f_1, \dots,
f_r$ such that $f_1 \leq f_2 \leq \dots
\leq f_r$. We consider $S$ as an $\NN$-graded ring and each $\deg x_i$ =
$1$.

The elements of the form $x^ae_i$, where $x^a =
x_1^{a_1} x_2^{a_2} \dots x_n^{a_n}$ for $a= (a_1, \dots, a_n)\in
\NN_0^n $, are called \textit{monomials} of $F$.

A monomial $me_i \in F$ is called a \textit{squarefree monomial} of $F$ if $m$ is a squarefree monomial of $S$, \textit{i.e.}, $m=x_{i_1}x_{i_2}\cdots x_{i_d}$
with $1\leq i_1<i_2< \cdots < i_d \leq n.$

For every monomial $me_i \in F$, we set
\[
\deg(m e_i) = \deg(m) + \deg(e_i).
\]

In particular if $F \simeq S^n$ and $e_i = (0, \dots, 0, 1,0,\dots, 0)$, where $1$ appears in the $i$-th place, one has
\[
\deg(m e_i) = \deg(m).
\]

A \textit{monomial submodule} $M$ of $F$ is a submodule generated by
monomials, \textit{i.e.}, $M = \oplus_{i=1}^r I_i e_i$, where $I_i$
are the monomial ideals of $S$ generated by those monomials $m$
of $S$ such that $me_i \in M$ \cite{Ei}.

A monomial submodule $M = \oplus_{i=1}^r I_i e_i \subsetneq F$ is a \textit{squarefree monomial submodule} if every $I_i$ is a squarefree monomial ideal of $S$, \textit{i.e.}, every $I_i$ is a monomial ideal of $S$ generated by squarefree monomials.

For every monomial ideal $I \varsubsetneq S$, we denote by $G(I)$ the
unique minimal set of monomial generators of $I$, by $G(I)_{\ell}$
the set of monomials $v$ of $G(I)$ such that $\deg v = \ell$ and by
$G(I)_{> \ell}$ the set of monomials $u$ of $G(I)$ such that $\deg u
> \ell$.

For every monomial submodule $M = \oplus_{i=1}^r I_i e_i$ of $F$ such that $I_i \subsetneq S$, for $i=1,\ldots,r$, we set
\begin{eqnarray*}
  G(M) &=& \{ue_i \,:\, u \in G(I_i),i = 1, \dots, r\}, \\
  G(M)_{\ell} &=& \{ue_i\, : \, u \in G(I_i)_{\ell - f_i }, \,i = 1,
\dots, r\}, \\
  G(M)_{> \ell} &=& \{ue_i \in G(M)\,:\, u \in G(I_i)_{> \ell - f_i }, \,i = 1, \dots, r\}.
\end{eqnarray*}

Let $M$ be a  finitely generated graded $S$-module, then $M$ has a
minimal graded free $S$-resolution
\[
F. : 0 \rightarrow F_s \rightarrow \cdots \rightarrow F_1
\rightarrow F_0 \rightarrow M \rightarrow 0,
\]
where $F_i = \oplus_{j \in \ZZ}S(-j)^{\beta_{i,j}}$. The
integers $\beta_{i,j} = \beta_{i,j}(M) = \textrm{dim}_K
\textrm{Tor}_i(K, M)_j $ are called the graded Betti numbers of
$M$.
\begin{Def} \label{def:extr} A Betti number $\beta_{k,k+\ell}(M) \neq 0$ is called {\it
extremal} if $\beta_{i,\, i+j}(M) = 0$ for all $i \geq k$, $j \geq
\ell$, $(i, j) \neq (k, \ell)$.
\end{Def}
The pair $(k, \ell)$ is called a corner.

\section{Squarefree monomial submodules}\label{classes}
In this Section, we analyze some classes of squarefree monomial submodules of the finitely generated graded free $S$-module $F =\oplus_{i=1}^r Se_i$.

If $I$ is a graded ideal of the polynomial ring $S$, we denote by
$\indeg I$ the \emph{initial degree} of $I$, \emph{i.e.}, the minimum $t$ such that $I_t \neq 0$.

For a monomial $1 \neq u \in S$, we set
 \[\supp(u)=\{i: x_i\,\, \textrm{divides}\,\, u\},\]
and \[\m(u) = \max \{i:i\in \supp(u)\}.\]
Moreover, we set $\m(1) = 0$.
\begin{Def} Let $I\subsetneq S$ be a squarefree monomial ideal. $I$  is called a \textit{squarefree stable ideal} if for all $u \in G(I)$ one has
$(x_j u)/x_{\m(u)} \in I$ for all $j < \m(u), j \notin \supp(u)$.\\
$I$  is called a \textit{squarefree strongly stable ideal} if for all $u \in G(I)$ one has
$(x_j u)/x_i \in I$ for all $i \in \supp(u)$ and all $j < i$, $j \notin \supp(u)$.
\end{Def}
For every $1\leq d\leq n$, let $[x_1, \ldots, x_n]^d$ be the squarefree monomial ideal of $S$ whose minimal system of monomial generators is given by all the degree $d$ squarefree monomials in the variables $x_1, \ldots, x_n$.

For example, for $n=4$ and $d=3$,
\[[x_1, \ldots, x_4]^3 = (x_1x_2x_3, x_1x_2x_4,x_1x_3x_4,x_2x_3x_4).\]


Following \cite{KP}, we give the following definition.
\begin{Def}\label{def:squarestable} A submodule $M$ of $F$ is {\it a squarefree stable submodule} if
\begin{enumerate}
\item[(1)] $M = \oplus_{i=1}^rI_i e_i$ is a squarefree monomial submodule with
$I_i \subsetneq S$, for $i=1,\ldots,r$;
\item[(2)]  for every squarefree monomial $ue_i \in M$, $u \in S$, then
$x_j\frac{u}{x_{\m(u)}}e_i \in M$, for all $j < \m(u)$, $j\notin \supp(u)$;
\item[(3)] $[x_1,\ldots,x_n]^{f_j-f_i}I_j \subseteq I_i$ for every $j>i.$
\end{enumerate}
\end{Def}
\begin{Prop} \label{newsquarestable} Let $M\subseteq F$ be a graded submodule.

Then $M$ is a squarefree stable submodule of $F$ if and
only if
\begin{enumerate}
\item[(i)] $M = \oplus_{i = 1}^r I_i e_i$, with $I_i\varsubsetneq S$
squarefree stable ideal, for $i=1,\ldots,r$, and
\item[(ii)] $[x_1, \ldots, x_n]^{f_{i+1}-f_i}I_{i+1} \subseteq I_i$, for $i = 1, \dots, r-1$.

\end{enumerate}
\end{Prop}
\begin{proof}
%
In order to prove the assert it is sufficient to show that condition (3) in Definition \ref{def:squarestable} is equivalent to the following inclusion:
\begin{equation}\label{equa1}
    [x_1, \ldots, x_n]^{f_{i+1}-f_i}I_{i+1} \subseteq I_i,
\end{equation}
for $i=1, \ldots, r-1$.

It is clear that statement (\ref{equa1}) follows from condition (3) for $j=i+1$.

Hence, in order to prove the required equivalence, it is sufficient to verify that
\[\mbox{$[x_1, \ldots, x_n]^{f_{i+t}-f_i}I_{i+t} \subseteq I_i$,\quad for $t\geq 1$}.\]

We proceed by induction on $t$.

For $t=1$ the assert follows from (\ref{equa1}).

Now take $t>1$ and suppose $[x_1, \ldots, x_n]^{f_{i+t}-f_i}I_{i+t} \subseteq I_i$.

Since
\[[x_1, \ldots, x_n]^{f_{i+t+1}-f_i}I_{i+t+1} = [x_1, \ldots, x_n]^{f_{i+t+1}-f_{i+t}+f_{i+t}-f_i}I_{i+t+1},\]
from (\ref{equa1}), we have
\[[x_1, \ldots, x_n]^{f_{i+t+1}-f_i}I_{i+t+1} \subseteq  [x_1, \ldots, x_n]^{f_{i+t}-f_i}I_{i+t},\]
and from the inductive hypothesis we get the stated result.
%
\end{proof}
Furthermore, we give the following definition.

\begin{Def}\label{def:squarestrostable} A submodule $M$ of $F$ is a {\it squarefree strongly stable} submodule if
\begin{enumerate}
\item[(1)] $M = \oplus_{i = 1}^r I_i e_i$, with $I_i\varsubsetneq S$
squarefree strongly stable ideal, for $i=1,\ldots,r$;
\item[(2)] $[x_1, \ldots, x_n]^{f_{i+1}-f_i}I_{i+1} \subseteq I_i$, for $i=1, \ldots, r-1$.
\end{enumerate}
\end{Def}

\begin{Rem} \em If $M$ is a graded submodule of the
finitely generated graded free $S$-module $S^r$, then $M$ is a squarefree (strongly) stable submodule of $S^r$ if and only if $M = \oplus_{i = 1}^r I_i e_i$, with $I_i\varsubsetneq S$
squarefree (strongly) stable ideal, for $i=1,\ldots,r$, and
$$I_r \subseteq I_{r-1} \subseteq \cdots \subseteq I_1.$$
\end{Rem}
For every $1\leq d\leq n$, let $\{x_1, \ldots, x_n\}^d$ be the set of all squarefree monomials of degree $d$ in the variables $x_1, \ldots, x_n$. We can order  $\{x_1, \ldots, x_n\}^d$ with the \textit{squarefree lexicographic order} $\geq_{\textrm{slex}}$ \cite{AHH2}.

Precisely, let
\[u=x_{i_1}x_{i_2}\cdots x_{i_d}, \qquad v=x_{j_1}x_{j_2}\cdots x_{j_d},\]
with $1\leq i_1< i_2< \cdots < i_d\leq n$, $1\leq j_1< j_2< \cdots < j_d\leq n$, be squarefree monomials of degree $d$ in $S$, then
\[\mbox{$u >_{\textrm{slex}} v$ \qquad if \qquad $i_1=j_1, \ldots, i_{s-1}=j_{s-1}$ \qquad and \qquad $i_s<j_s$}, \]
for some $1 \leq s \leq d$.

A nonempty set $L \subseteq \{x_1, \ldots, x_n\}^d$ is called a \textit{squarefree lexsegment set} of degree $d$ if for $u \in L$, $v \in \{x_1, \ldots, x_n\}^d$ such that $v >_{\textrm{slex}} u$, then $v \in L$.

\begin{Def} Let $I\subsetneq S$ be a graded ideal. $I$ is a \textit{squarefree lexsegment ideal of degree} $d$ if $I$ is generated by the squarefree monomials belonging to a
squarefree lexsegment set of degree $d$.\\
$I$ is a \textit{squarefree lexsegment ideal} if for all $1\leq d \leq n$, if $u, v \in \{x_1, \ldots, x_n\}^d$ with $u \in I$ and $v >_{\textrm{slex}} u$, then $v \in I$.
\end{Def}
For every $d\geq 1$, let $F_d$ be the part of degree $d$ of $F\oplus_{i=1}^r Se_i$. Denote by $M^s(F_d)$ the set of all squarefree monomials
of degree $d$ of $F$. We order such set using the ordering $>_{\textrm{slex}}$, above defined.  We will denote the new ordering by $>_{\textrm{slex}_F}$.

It is defined as follows: if $u e_i$ and $v e_j$ are squarefree monomials of $F$ such that $\deg(u e_i)=\deg(v e_j)$, then
\[
u e_i >_{\textrm{slex}_F} v e_j \qquad \textrm{if} \qquad
\left\{
\begin{array}{ll} i < j \,\,\, \mbox{or} \\
i = j \,\, \mbox{and}\,\, u >_{\textrm{slex}} v.
\end{array}
\right.
\]

\begin{Def} Let $\mathcal{L} = \oplus_{i = 1}^r I_i e_i$ be a  squarefree monomial submodule
of $F$ such that $I_i \varsubsetneq S$, for $i=1,\ldots,r$. We call $\mathcal{L}$ a \textit{squarefree lexicographic
submodule} if for each degree $d\geq 1$, if $u, v \in M^s(\mathcal{L}_d)$ with $u \in
\mathcal{L}$ and $v>_{\textrm{slex}_F}u$, then $v \in \mathcal{L}$.
\end{Def}
\begin{Prop} \label{lex} Let $\mathcal{L}\subsetneq S$ be a graded submodule.

Then $\mathcal{L}$ is a  squarefree lexicographic submodule of $F$ if and
only if
\begin{enumerate}
\item[(i)] $\mathcal{L} = \oplus_{i = 1}^r I_i e_i$, with $I_i\varsubsetneq S$
 squarefree lexsegment ideal, for $i=1,\ldots,r$, and
\item[(ii)] $[x_1,\dots, x_n]^{\rho_i + f_i - f_{i-1}}
\subseteq I_{i-1}$, for $i = 2, \dots, r$, with $\rho_i =
\indeg I_i$.

\end{enumerate}
\end{Prop}
\begin{proof} Let $\mathcal{L}$ be a squarefree lexicographic submodule of $F$.
\par\noindent (i) Since $\mathcal{L}$ is a squarefree monomial submodule of
$F$, one has $\mathcal{L} = \oplus_{i = 1}^r I_i e_i$, with $I_i$
squarefree monomial ideal of $S$, for every $i$. Let $u, v \in \{x_1, \ldots, x_n\}^d$ with $u \in I_i$ and $v>_{\textrm{slex}}u$. It follows
that $ve_i >_{\textrm{slex}_F}ue_i$. Since  $ue_i\in I_ie_i$ and
$\mathcal{L}$ is a squarefree lexicographic submodule of $F$, $ve_i \in
I_ie_i$, and so $v \in I_i$, \textit{i.e.}, $I_i$ is a squarefree lexsegment ideal of
$S$ for every $i$.
\par\medskip\noindent
(ii) Since $I_i$ is a squarefree lexsegment ideal of $S$, then
$x_1x_2\cdots x_{\rho_i}\in I_i$, $\rho_i = \indeg I_i$, and consequently
$x_1x_2\cdots x_{\rho_i}e_i\in I_ie_i$. On the other hand, $\mathcal{L}$ is a squarefree
lexicographic submodule of $F$, then for all $u\in \{x_1, \ldots, x_n\}^{\rho_i + f_i - f_{i-1}}$, we have that $ue_{i-1} >_{\textrm{slex}_F}
x_1x_2\cdots x_{\rho_i}e_i$. Hence, $ue_{i-1} \in
I_{i-1}e_{i-1}$, \textit{i.e.}, $u \in I_{i-1}$.

Conversely, let $\mathcal{L}$ be a graded submodule of $F$
satisfying (i) and (ii).
\par\noindent
Since every ideal $I_i$ is a squarefree lexsegment ideal, we have only to
prove that for any pair $(i, j)$ of integers with $1 \leq i < j
\leq r$, if $ue_i, ve_j\in M^s(\mathcal{L}_d)$, then $ve_j \in \mathcal{L}$ implies $ue_i \in \mathcal{L}$, where $M^s(\mathcal{L}_d)$ is the set of
all squarefree monomials of degree $d$ of $\mathcal{L}$.
\medskip\par\noindent
(Case 1). $i = j-1$. Let $ue_{j-1}, ve_j\in M^s(\mathcal{L}_d)$ with $ve_j \in \mathcal{L}$.

Since $d = \deg ue_{j-1} = \deg ve_j$, it follows that
$\deg u = \deg v + f_j - f_{j-1} \geq \rho_j + f_j - f_{j-1}$ and
so $u \in [x_1,\dots, x_n]^{\rho_j + f_j - f_{j-1}} \subseteq
I_{j-1}$.
\medskip\par\noindent
(Case 2). $i \leq j-2$. Let $ue_i >_{\textrm{slex}_F} ve_j$ with $ue_i, ve_j \in M^s(\mathcal{L}_d)$ and
$ve_j \in \mathcal{L}$.

For $t = i+1, \ldots, j-1$, set $w_t = x_1x_2\cdots x_{d-f_t}$.
\par\noindent
It is
\[ue_i >_{\textrm{slex}_F} w_{i+1}e_{i+1} >_{\textrm{slex}_F}  w_{i+2}e_{i+2} >_{\textrm{slex}_F} \cdots  >_{\textrm{slex}_F}
 w_{j-1}e_{j-1}>_{\textrm{slex}_F} ve_j.\]
Since $d = \deg ue_i = \deg  w_te_t =\deg ve_j$, for $t =
i+1, \ldots, j-1$, then, from (Case 1), $ w_{j-1} \in I_{j-1}$, $
w_{j-2}\in I_{j-2}$, $\ldots$, $ w_{i+1} \in I_{i+1}$ and finally
$u \in I_i$.
\end{proof}

\begin{Expl} (1) Let $S = K[x_1, x_2,x_3,x_4]$ and $F \simeq S^2$, $e_1 = (1, 0)$
and $e_2 = (0, 1)$. The submodule of $F$
\[\mathcal{L} = (x_1x_2, x_1x_3)e_1 \oplus (x_1x_2x_3, x_1x_2x_4)e_2\] is not a squarefree
lexicographic submodule of $F$ even if the ideals $(x_1x_2, x_1x_3)$, $(x_1x_2x_3, x_1x_2x_4)$ are squarefree lexsegment ideals of $S$. In fact, $x_1x_2x_3e_2 \in
\mathcal{L}_3$ but $x_2x_3x_4e_1 >_{\textrm{slex}_F} x_1x_2x_3e_2$ and $x_2x_3x_4e_1 \notin
\mathcal{L}_3$. Observe that $[x_1,x_2,x_3,x_4]^3 \nsubseteq (x_1x_2, x_1x_3)$.
\medskip
\par\noindent (2) Let $S = K[x_1, x_2,x_3,x_4,x_5]$ and $F \simeq S^3$, $e_1 = (1, 0,0)$,
$e_2 = (0, 1,0)$ and $e_3 = (0,0,1)$. The submodule of $F$
\[
\begin{aligned}
\mathcal{L} = [x_1 ,x_2, x_3,x_4,x_5]^2 e_1 &\oplus (x_1x_2x_3, x_1x_2x_4, x_1x_2x_5, x_1x_3x_4, x_2x_3x_4x_5)e_2 \ \oplus \\
&\oplus (x_1x_2x_3x_4, x_1x_2x_3x_5, x_1x_2x_4x_5)e_3
\end{aligned}
\]
is a squarefree lexicographic submodule of $F$.
\end{Expl}

\begin{Cor} \label{degree} Let $\mathcal{L} = \oplus_{i = 1}^r I_i e_i \subsetneq F$ be a squarefree lexicographic submodule.

Set
\[\mu_i = \max\{\deg ue_i\,:\, ue_i \in G(I_ie_i)\},\]
for $ i=1, \dots,r.$ Then
\begin{enumerate}
\item[1)] $\mu_i - f_i \leq \indeg I_{i+1} + f_{i+1} -f_i $,  for $i=1, \dots, r-1$.
\item[2)] $\mu_1 \leq \mu_2 \leq \cdots \leq \mu_r $.
\end{enumerate}
\end{Cor}
\begin{proof} It is a consequence of the fact that if
$\mathcal{L}$ is a squarefree lexicographic submodule then $[x_1,\dots,
x_n]^{\rho_i + f_i - f_{i-1}} \subseteq I_{i-1}$, for $i =
2, \dots, r$, with $\rho_i = \indeg (I_i)$.
\end{proof}
The next result shows the relation between the class of squarefree lexicographic submodules and the class of squarefree (strongly) stable submodule.
\begin{Prop} Let $\mathcal{L} = \oplus_{i = 1}^r I_i e_i \subsetneq F$ be a  squarefree lexicographic submodule, then $\mathcal{L}$ is a  squarefree strongly stable submodule.
\end{Prop}
\begin{proof} Since every  squarefree lexsegment ideal is a squarefree strongly stable ideal \cite{AHH2}, from Proposition \ref{newsquarestable}, it is sufficient to show that
\begin{equation}\label{lex1}
    [x_1, \ldots, x_n]^{f_{i+1}-f_i}I_{i+1} \subseteq I_i
\end{equation}
for $i=1, \ldots, r-1$.

Since
\[[x_1, \ldots, x_n]^{f_{i+1}-f_i}I_{i+1} \subseteq [x_1, \ldots, x_n]^{f_{i+1}-f_i+\indeg I_{i+1}},\]
the assert follows from Proposition \ref{lex}.
\end{proof}
Hence, we have the following hierarchy of submodules:\\
squarefree lexicographic submodule $\Rightarrow$ squarefree strongly stable submodule $\Rightarrow$ squarefree stable submodule.

\section{Extremal Betti numbers of squarefree stable submodules}\label{extr}
In this Section we examine the extremal Betti numbers of squarefree stable submodules.

If  $I$ is a squarefree stable ideal, then \cite{AHH2}:
\begin{equation}\label{AHHeq}
    \beta_{i, \, i+j}(I) =\sum_{u \in G(I)_j} \binom{\m(u)-j}{i}.
\end{equation}

Hence, since every squarefree stable submodule $M \subsetneq F$ is a squarefree
monomial submodule, we have that
\begin{equation}\label{betti2}
\beta_{k,\,k+j}(M) = \sum_{i=1}^r \beta_{k,\,{k+j - f_i} }(I_i)=\sum_{i=1}^{r}\left[\sum_{u \in G(I_i)_{j-f_i}}\binom{\m(u)-j+f_i}{k}\right].
\end{equation}


The next result shows that all linear strands of a squarefree stable submodule begin in homological degree $0$.

For a positive integer $t$, we set $[t] = \{1, \ldots, t\}$.
\begin{Lem} \label{0} Let $M = \oplus_{t=1}^rI_te_t \subsetneq F$ be a squarefree stable
submodule. If $\beta_{i, i+j}(M) \neq 0$, then $\beta_{k, k+j}(M)
\neq 0$ for $k=0,\ldots, i$.
\end{Lem}
\begin{proof} If $\beta_{i, i+j}(M)\neq 0$, by (\ref{betti2})
there exists $t \in [r] $ and $u \in G(I_t)_{j-f_t}$
such that $\m(u) \geq i+j-f_t$.

It follows that $\m(u) \geq k+j-f_t$,
for $k=0, \dots, i$. Hence, $\beta_{k, k+j-f_t}(I_t) \neq 0$ and
$\beta_{k, k+j}(M) \neq 0$, for $k=0,\dots, i$.
\end{proof}
From Definition \ref{def:extr}, it follows:
\begin{Cor}\label{cor:varie} Under the same hypotheses of Lemma \ref{0} for $M$. The
following conditions are equivalent:
\begin{enumerate}
\item[\rm{(a)}] $\beta_{k, k+\ell}(M)$ is extremal;
\item[\rm{(b)}] \begin{enumerate}
\item[\rm{(1)}] $\beta_{k, k+\ell}(M) \neq 0$;
\item[\rm{(2)}] $\beta_{k, k+j}(M) = 0$, for $j > \ell$;
\item[\rm{(3)}] $\beta_{i, i+\ell}(M) = 0$, for $i >k$.
\end{enumerate}
\end{enumerate}
\end{Cor}
\begin{Thm}
\label{equiv} Let $M = \oplus_{i=1}^rI_ie_i \subsetneq F$ be a  squarefree stable submodule.
\par The following conditions are equivalent:
\begin{enumerate}
\item[\rm{(1)}] $\beta_{k, \, k+ \ell}(M)$ is extremal;
\item[\rm{(2)}] $k + \ell = \max\{\m(u)+f_i \, :\, ue_i \in G(M)_{\ell},\,\,
i=1, \ldots, r\}$ and $\m(u)+f_i < k+j$, for all $j > \ell$ and
for all $ue_i \in G(M)_j$.
\end{enumerate}
\end{Thm}
\begin{proof} (1)$\Rightarrow$ (2). By (\ref{betti2}) $\beta_{k, \, k+ \ell}(M)\neq 0$ if and only if there exists a squarefree
monomial $ue_i \in G(M)_{\ell}$ such
that $\m(u)+f_i \geq k+\ell$, for some $i\in [r]$. As a consequence, $\max\{\m(u)+f_i \, :\, ue_i \in G(M)_{\ell}\,,
i=1, \ldots, r\}\geq k +\ell$.

Suppose $t+\ell :=\max\{\m(u)+f_i \, :\, ue_i \in G(M)_{\ell}\,,
i=1, \ldots, r\}> k+\ell$. Hence, $\beta_{t, t +\ell}(M) \neq 0$, for $t>k$. This is a contradiction from Corollary \ref{cor:varie}, (b), whereupon
\[k + \ell = \max\{\m(u)+f_i \, :\, ue_i \in G(M)_{\ell}\,,
i=1, \ldots, r\}.\]

Suppose there exists $j>\ell$ and a squarefree monomial $ue_i\in G(M)_j$, for some $i\in [r]$, such that $\m(u)+f_i \geq k+j$. From (\ref{betti2}), then $\beta_{k, \, k+ j}(M)\neq 0$. Again a contradiction from Corollary \ref{cor:varie}, (b).

(2)$\Rightarrow$ (1). Since $k + \ell = \max\{\m(u)+f_i \, :\, ue_i \in G(M)_{\ell}\,,
i=1, \ldots, r\}$, then $\beta_{k, \, k+ \ell}(M)\neq 0$ and $\beta_{i, \, i+ \ell}(M)= 0$, for all $i>k$. On the other hand, $\m(u)+f_i < k+j$, for all $j > \ell$ and for all $ue_i \in G(M)_j$, implies $\beta_{k, \, k+ j}(M)= 0$. Therefore, from Corollary \ref{cor:varie}, we get the assert.
\end{proof}
As  consequences, we obtain the following corollaries.
\begin{Cor} \label{cor:uniq} Let $M = \oplus_{i=1}^rI_ie_i \subsetneq F$ be a squarefree stable
submodule and let $\beta_{k, \, k+ \ell}(M)$ an extremal Betti
number of $M$. Then
\[\beta_{k, \, k+ \ell}(M) = \vert\{ue_i \in G(M)_{\ell}\,:\, \m(u) + f_i=
k+\ell,\, i=1, \ldots, r\}\vert.\]


\end{Cor}
\begin{Cor} \label{unique} Let $M = \oplus_{i=1}^rI_ie_i \subsetneq F$ be a squarefree stable
submodule.

Set $$\ell = \max\{j:G(M)_j \neq \emptyset\}$$ and
$$m = \max\{m(u)+f_i\,:\, ue_i \in G(M), i=1, \ldots, r\}.$$
Then $\beta_{m-\ell,\,m}$ is the unique extremal Betti number
of $M$ if and only if
\[m = \max\{m(u)+f_i\,:\, ue_i \in G(M)_{\ell}, \,i=1, \ldots, r\},\]
and for every $w\in G(M)_j$, $j<\ell$, $\m(w) < m$.
\end{Cor}
%
\begin{Rem} \em \label{rem:free} Under the same hypotheses of Theorem \ref{equiv}, if $F\simeq S^r$, $\beta_{k, \, k+ \ell}(M)$ is extremal if and only if
\[k + \ell = \max\{\m(u)\, :\, ue_i \in G(M)_{\ell}\,,
i=1, \ldots, r\},
\]
and $\m(u) < k+j$, for all $j > \ell$ and for all $ue_i \in G(M)_j$.

Moreover,
\[\beta_{k, \, k+ \ell}(M) = \vert\{ue_i \in G(M)_{\ell}\,:\, \m(u)=
k+\ell,\, i=1, \ldots, r\}\vert.\]
\end{Rem}

\begin{Rem}\em
If $I$ is a squarefree stable monomial ideal of $S$ and $\beta_{k, k+\ell}(I)$ is an extremal Betti number of $I$, then from (\ref{betti2}) and Remark \ref{rem:free}, we have
\begin{equation}\label{diseq1}
    1\leq \beta_{k, k+\ell}(I) \leq \binom{k+\ell-1}{\ell-1}.
\end{equation}

In fact, there are exactly $\binom{k+\ell-1}{\ell-1}$ squarefree monomials of degree $\ell$ in $S$ with $\m(u) = k+\ell$.

If $M = \oplus_{i=1}^rI_ie_i \varsubsetneq F$ is a squarefree stable
submodule and $\beta_{k, \, k+ \ell}(M)$ is an extremal Betti
number of $M$, then there exist some squarefree stable ideals $I_{j_1},
\ldots, I_{j_t}$, $\{j_1, \ldots, j_t\}\subseteq [n]$, $ 1 \leq j_1 < j_2 < \cdots < j_t \leq
r$, with $(k, \ell -f_{j_1}), \ldots, (k, \ell -f_{j_t})$ as corners.

Then
\begin{equation*}
    \beta_{k, \, k+ \ell}(M) = \sum_{i=1}^t \beta_{k, \, k+
\ell}(I_{j_i}e_{j_i}) = \sum_{i=1}^t \beta_{k, \, k+ \ell-
f_{j_i} }(I_{j_i}),
\end{equation*}
and
\begin{equation*}
    1 \leq \beta_{k, \, k+ \ell}(M) \leq \sum_{i=1}^t \binom{k + \ell -
f_{j_i} -1}{\ell- f_{j_i} -1},
\end{equation*}
where $\binom{k + \ell - f_{j_i}
-1}{\ell - f_{j_i} -1}$ is the number of all squarefree monomials $u$ of $S$
of degree $\ell - f_{j_i}$ with $\m(u) =k+\ell$.
\end{Rem}

\section{A criterion for determing extremal Betti numbers} \label{crit}
In this Section, we describe a criterion for determining the extremal Betti numbers of a squarefree stable submodule of the graded free $S$-module $S^r$, $r\geq 1$.

Let $M=\oplus_{i=1}^rI_ie_i \subsetneq S^r $ be a squarefree stable submodule.

If $M$ is generated in one degree $d$, then $M$ has a unique extremal Betti number $\beta_{m-d,\,m}(M)$, where $m= \max\{ \m(u)\,:\, ue_i \in G(M)
, \, i=1, \ldots, r \}$.

Let $M$ be generated in degrees $1\leq d_1 < d_2 < \cdots < d_t \leq n$.

Set
\[m_{d_j} = \max\{\m(u)\,:\, ue_i \in G(M)_{d_j}, \, i=1, \ldots, r\},\]
for $j=1, \ldots, t$.

Consider the following sequence of non negative integers associated to $M$:
\begin{equation}\label{degseq1}
   \bd(M) =(m_{d_1}-d_1, m_{d_2}-d_2, \ldots, m_{d_t}-d_t).
\end{equation}
We call it the \textit{degree-sequence} of $M$.

\begin{Rem} \em Assume that for some $j \in [t]$, $G(M)_{d_j} = \{x_1\cdots x_{d_j}\}$. Then $m_{d_j}-d_j=0$.

If $F=S$, and consequently $M$ is a squarefree monomial ideal in $S$, then
$m_{d_i}-d_i>0$, for $i=2, \ldots, t$. Moreover, if $G(M)_{d_1}=G(M)_{\indeg(M)} = \{x_1\cdots x_{\indeg(M)}\}$, then $m_{d_1}-d_1=0$.
\end{Rem}
We can observe that, if
\begin{equation}\label{disdegree}
m_{d_1}-d_1 > m_{d_2}-d_2> \cdots > m_{d_t}-d_t,
\end{equation}
then, from Theorem \ref{equiv}, for $i=1, \ldots, t$, $\beta_{m_{d_i}-d_i,\,m_{d_i}}(M)$ is an extremal Betti number of $M$.

Suppose that (\ref{disdegree}) does not hold.

We construct a suitable subsequence of the \textit{degree-sequence} $\bd(M)$:
\begin{equation}\label{subseq}
    \widehat{\bd(M)}=(m_{d_{i_1}}-d_{i_1}, m_{d_{i_2}}-d_{i_2}, \ldots, m_{d_{i_q}}-d_{i_q}),
\end{equation}
with $d_1\leq d_{i_1} < d_{i_2}< \cdots <d_{i_q}= d_t$ and such that, for $j=1, \ldots, q$, $\beta_{m_{d_{i_j}}-d_{i_j},\,m_{d_{i_j}}}(M)$ is an extremal Betti number of $M$.
\begin{Constr} \label{sequence} The admissible subsequence $\widehat{\bd(M)}$ is obtained as follows:
\begin{enumerate}
\item[] $d_{i_1}$ is the smallest degree of a monomial generator of $M$ such that
\[\mbox{$m_{d_{i_1}}-d_{i_1} > m_{d_j}-d_j$, for all $j>i_1$};\]
\item[] $d_{i_2}>d_{i_1}$ is the smallest degree of a monomial generator of $M$ such that
\[\mbox{$m_{d_{i_2}}-d_{i_2} > m_{d_j}-d_j$, for all $j>i_2 >i_1$};\]
\item[] $\ldots \ldots$\\
\item[] $d_{i_{q-1}}>d_{i_{q-2}}$ is the smallest degree of a monomial generator of $M$ such that
\[\mbox{$m_{d_{i_{q-1}}}-d_{i_{q-1}} > m_{d_j}-d_j$, for all $j>i_{q-1}> \cdots >i_1$};\]
\item[] $d_{i_q} = d_t$.
\end{enumerate}

\medskip

It is $m_{d_{i_1}}-d_{i_1} > m_{d_{i_2}}-d_{i_2} > \cdots > m_{d_{i_q}}-d_{i_q}$, and Theorem \ref{equiv} guarantees that
\[\beta_{m_{d_{i_1}}-d_{i_1},\,m_{d_{i_1}}}(M), \beta_{m_{d_{i_2}}-d_{i_2},\,m_{d_{i_2}}}(M),\ldots, \beta_{m_{d_{i_q}}-d_{i_q},\,m_{d_{i_q}}}(M)\]
are extremal Betti numbers of $M$.

The integer $q$ is called the \textit{degree-length} of $M$ and
gives the number of the extremal Betti numbers of the squarefree stable submodule $M$. We will denote such integer by $\bl(M)$.
\end{Constr}
\begin{Expls} (1) Let $S= K[x_1, x_2, x_3, x_4,x_5, x_6, x_7]$ and let
\[I = (x_1x_2, x_1x_3,x_1x_4,x_1x_5,x_1x_6,x_2x_3x_4, x_2x_3x_5,x_2x_4x_5, x_2x_3x_6x_7, x_3x_4x_5x_6x_7)\]
be a squarefree strongly stable ideal of $S$.

The \textit{degree-sequence} of $I$ is
\begin{equation*}
    \bd(I)=(m_2-2,m_3-3,m_4-4,m_5-5)=(4,2,3,2).
\end{equation*}

Following Construction \ref{sequence}, the admissible subsequence $\widehat{\bd(I)}$ is $(4,3,2)$ and $\bl(I)=3$. The extremal Betti numbers of $I$ are
\[\beta_{6-2,6}(I), \beta_{7-4,7}(I),\beta_{7-5,7}(I),\] as the Betti table of $I$ shows:
\[
\begin{array}{lllllll}
    &   & 0 & 1 & 2 & 3 & 4 \\
\hline
  2 & : & 5 & 10 & 10 & 5 & 1 \\
  3 & : & 3 & 5 & 2 & - & -   \\
  4 & : & 1 & 3 & 3 & 1 & - \\
  5 & : & 1 & 2 & 1 & - & -
\end{array}
\]
\par\noindent
(2) Let $S = K[x_1, x_2,x_3,x_4,x_5,x_6]$ and $F \simeq S^4$, $e_1 = (1, 0, 0, 0)$,
$e_2 = (0, 1, 0, 0)$, $e_3 = (0, 0, 1, 0)$, $e_4 = (0, 0, 0, 1)$. Let
\[M = I_1e_1 \oplus I_2e_2 \oplus I_3e_3 \oplus I_4e_4\]
with
\[I_1 = (x_1x_2,x_1x_3,x_1x_4, x_2x_3), \]
\[I_2 = (x_1x_2x_3, x_1x_2x_4, x_1x_3x_4, x_1x_3x_5, x_1x_3x_6, x_2x_3x_4, x_2x_3x_5,x_2x_3x_6),\]
\[I_3 = (x_1x_2x_3x_4, x_1x_2x_3x_5, x_1x_2x_4x_5, x_1x_2x_4x_6, x_2x_3x_4x_5), I_4 = (x_1x_2x_3x_4x_5),\]
be a squarefree stable submodule of $F$.

The \textit{degree-sequence} of $M$ is
\begin{equation*}\label{es2}
   \bd(M)= (m_2-2,m_3-3,m_4-4, m_5-5)=(2,3,2,0).
\end{equation*}

Following Construction \ref{sequence}, the admissible subsequence $\widehat{\bd(M)}$ is $(3,2,0)$ and $\bl(M)=3$. Indeed, the extremal Betti numbers of $M$ are
\[\beta_{6-3,6}(M), \beta_{6-4,6}(M), \beta_{5-5,5}(M),\] as one can read on its Betti table:

\[\begin{array}{llllll}
    &   & 0 & 1 & 2 & 3 \\
\hline
  2 & : & 4 &4 & 1 & - \\
  3 & : & 8 & 13 & 8 & 2    \\
  4 & : & 5 & 5 & 1 & - \\
  5 & : & 1 & -& -& -
\end{array}\]
\end{Expls}

\section{An application}\label{appl}
In this Section, we consider some special extremal Betti numbers of squarefree lexicographic submodules.

If $M$ is a finitely generated graded $S=K[x_1, \ldots, x_n]$-module, then $\beta_{i,\,j}(M)=0$, for all $i$ and $j>n$. Therefore, $\beta_{i,\,n}(M)$ is an extremal Betti number if
$\beta_{i,\,n}(M)\neq 0$. Such extremal Betti numbers are called \textit{super extremal} \cite{AHH3}.

The pair $(i, n-i)$ is called a \textit{super corner}.


Let $M = \oplus_{i = 1}^r I_i e_i\subsetneq
S^r$, $r\geq 1$, be a squarefree stable submodule and $G(M)$ its minimal system of monomial generators.

Set
\[b_i=\vert \{ue_j \in G(M)_{n-i}\,:\,\m(u)=n, \, j=1, \ldots, r\}\vert,\]
then $\beta_{i,\,n}(M)=b_i$, for $i=1, \ldots, n-1$. We call
\[\bb(M) = (b_0, b_1, \ldots, b_{n-1})\]
the $\bb$-vector of $M$.

For $\bb(I) = (b_0, b_1, \ldots, b_{n-1})$, we define the \textit{support} of $\bb(M)$ to be the following set:
\[\supp(\bb(M))=\{i \in \{0, 1, \ldots,n-1\}\,:\,b_i\neq 0\}.\]

For every $u\in \{x_1, \ldots, x_n\}^d$, we denote by $\langle u \rangle$ the squarefree lexsegment ideal of degree $d$ in $S$ defined as follows:
\[\langle u \rangle= (w\in \{x_1, \ldots, x_n\}^d\,:\, w \geq_{\textrm{slex}} u).\]
\begin{Lem} \label{char} Let $n\geq 3$ and $X$ a non-empty subset of $\{0, 1, \ldots,n-1\}$.

For every integer $r \geq \vert X\vert$ there exists a squarefree lexicographic submodule $\mathcal{L}\subsetneq S^r$ with $\supp(\bb(\mathcal{L}))=X$.
    \end{Lem}
    \begin{proof} Set $X = \{k_1, k_2, \ldots, k_t\}$ with $k_1 < k_2 < \cdots < k_t$ and $r \geq t$.

By Proposition \ref{lex}, an admissible squarefree lexicographic submodule $\mathcal{L}\subsetneq S^r$ such that $\supp(\bb(\mathcal{L}))=X$ is:
\[\mathcal{L} = \oplus_{j=0}^{t-2}[x_1,\ldots,x_n]^{n-k_{t-j}}e_{j+1} \oplus\left(\oplus_{i=t}^{r-1}[x_1,\ldots,x_n]^{n-k_1}e_i\right)\oplus \langle x_1x_2\cdots x_{n-k_1-1}x_n \rangle e_r.\]

Note that  $\bb(\mathcal{L}) = (b_0, \ldots, b_{n-1})$, where $b_i = 0$, for $i\in \{0, 1, \ldots,n-1\} \setminus X$, $b_{k_j} = \binom{n-1}{n-k_j-1}$, for $j=2, \ldots, t$ and $b_{k_1} = 1+(r-t)\binom{n-1}{n-k_1-1}$.

\end{proof}
\begin{Prop} Given three integers $n$, $t$, $r$ such that $n\geq 3$, $1\leq t\leq n-1$, $r\geq t$,
$t$ pairs of integers $(k_1, \ell_1), (k_2, \ell_2), \ldots, (k_t, \ell_t)$,
with
\[0\leq k_t < k_{t-1} < \cdots < k_1\leq n-1,\qquad 1 \leq \ell_1 < \ell_2 < \cdots < \ell_t\leq n \]
and such that, for $i=1,\ldots,t$, $k_i+\ell_i=n$.

Then there exists a squarefree lexicographic submodule $\mathcal{L}\subsetneq S^r$, generated in degrees $\ell_1, \ell_2, \ldots, \ell_t$ with
$(k_1, \ell_1), (k_2, \ell_2),\ldots,(k_t, \ell_t)$ as super corners.
\end{Prop}
\begin{proof} Set $X = \{n-{\ell_t}, n-{\ell_{t-1}},\ldots, n-{\ell_1}\} = \{k_t, k_{t-1}, \ldots, k_1\}$. One has that $X$ is a non-empty subset of $\{0, 1, \ldots,n-1\}$ and from
Lemma \ref{char} the squarefree monomial submodule of $S^r$:
\[
\mathcal{L} = \oplus_{i=1}^{t-1}[x_1,\ldots,x_n]^{\ell_i}e_i \oplus\left(\oplus_{i=t+1}^{r-1}[x_1,\ldots,x_n]^{\ell_t}e_i\right)
\oplus \langle x_1x_2\cdots x_{\ell_t-1}x_n \rangle e_r
\]
is a squarefree lexicographic submodule generated in degrees $\ell_1, \ell_2, \ldots, \ell_t$ and such that $\supp(\bb(\mathcal{L})) = X$.
Thus, $b_{k_i} = \beta_{k_i, n}$ are super extremal Betti numbers of $\mathcal{L}$, and the assert follows.
\end{proof}

\begin{Rem}\em Let $n=2$. If one considers the pair $(k, \ell) = (0,2)$, then there exists the squarefree lexicographic submodule $\mathcal{L} = \oplus_{i=1}^r(x_1x_2)e_i$ with $\bb(\mathcal{L})=(r,0)$.

If one considers $(k, \ell) = (1,1)$, then there exists the squarefree lexicographic submodule $\mathcal{L} = \oplus_{i=1}^r(x_1,x_2)e_i$ with $\bb(\mathcal{L})=(0, r)$.

\end{Rem}


\begin{thebibliography}{99}

\bibitem{AHH2} {A. Aramova, J. Herzog and  T. Hibi}, \newblock Squarefree lexsegment ideals, \newblock Math.Z. {\bf 228} (1998), 353--378.
\bibitem{AHH3} {A. Aramova, J. Herzog and  T. Hibi}, \newblock Shifting operations and graded Betti numbers ideals, \newblock J. Algebraic Combin. {\bf 12} (2000), 207-222.
\bibitem{BCP} D.\ Bayer, H.\ Charalambous, and S.\ Popescu. Extremal Betti
numbers and Applications to Monomial Ideals, J. Algebra {\bf 221}
(1999), 497--512.

\bibitem{BH} W.\ Bruns, J.\ Herzog. {\em Cohen-Macaulay
rings} {\bf 39} Revised Edition, Cambridge 1996.

\bibitem{CNR}A.Capani, A. Niesi, L.Robbiano, \textit{CoCoA: A system for doing computations in Commutative Algebra,} available
at http://cocoa.dima.unige.it.

\bibitem{CL} M.\ Crupi, M.\ La Barbiera. Algebraic properties of universal squarefree lexsegment ideals, Algebra Colloq., to appear (2013).
\bibitem{CU1} M.\ Crupi, R.\ Utano. Extremal Betti numbers of lexsegment ideals, Lecture
notes in Pure and Applied Math., Geometric and combinatorial
aspects of Commutative algebra, {\bf 217}, (2000), 159--164.

\bibitem{CU2} M.\ Crupi, R.\ Utano. Extremal Betti numbers of graded ideals, Results Math.
{\bf 43}, (2003), 235-244.

\bibitem{CU3} M.\ Crupi, R.\ Utano. Minimal resolutions of some monomial submodules, Results Math.
{\bf 55}, (2009), 311-328.

\bibitem{CR} M.\ Crupi, G.\ Restuccia. Monomial modules and graded Betti numbers, Math. Notes
{\bf 85}, (2009), 690-702.

\bibitem{EK} S. Eliahou, M. Kervaire. Minimal resolutions of some
monomial ideals, J. Algebra, {\bf 129} (1990), 1--25.

\bibitem{Ei} D.\ Eisenbud, {\em Commutative Algebra with a
view toward  Algebraic Geometry,}, Springer-Verlag, 1995.

\bibitem{GDS} {D. R. Grayson and M.E. Stillman}, {\em Macaulay2, a software system for research in algebraic geometry}. Available at http://www.math.uiuc.edu/Macaulay2.

\bibitem{KP} K.\ Pardue.  Deformation classes of graded modules and maximal Betti numbers
Illinois J. Math. {\bf 40}, N. 4 (1996), 564--585.


\end{thebibliography}
\enddocument